\documentclass[sn-mathphys-ay]{sn-jnl}

\usepackage{graphicx}%
\usepackage{multirow}%
\usepackage{amsmath,amssymb,amsfonts}%
\usepackage{amsthm}%
\usepackage{mathrsfs}%
\usepackage[title]{appendix}%
\usepackage{xcolor}%
\usepackage{textcomp}%
\usepackage{manyfoot}%
\usepackage{booktabs}%
\usepackage{algorithm}%
\usepackage{algorithmicx}%
\usepackage{algpseudocode}%
\usepackage{listings}%

\theoremstyle{thmstyleone}%
\newtheorem{theorem}{Theorem}
\newtheorem{proposition}[theorem]{Proposition}%
\newtheorem{lemma}[theorem]{Lemma}

 \newtheorem{remark}{Remark}%

\graphicspath{{./Figures-final/}}

\usepackage{textcomp}
\usepackage[T1]{fontenc}
\usepackage{pgfplots}
\usepgfplotslibrary{fillbetween}
\pgfplotsset{compat=1.10}
\usepackage{subcaption}

\definecolor{mycolor}{gray}{0.1} 
\definecolor{myblue}{rgb}{0.00000,0.44700,0.74100}%
\definecolor{myred}{rgb}{0.85000,0.32500,0.09800}%
\definecolor{myyellow}{rgb}{0.92900,0.69400,0.12500}%
\definecolor{mycolor4}{rgb}{0.49400,0.18400,0.55600}%

\def\r0{\mathcal{R}_0}
\newcommand{\dd}{\mathop{}\!\mathrm{d}}

\begin{document}

\title{Bistability and complex bifurcation diagrams generated by waning and boosting of immunity}

\date{\today}


\author*[1,2]{\fnm{Francesca} \sur{Scarabel}}\email{f.scarabel@leeds.ac.uk}
\equalcont{These authors contributed equally to this work.}

\author*[2,3]{\fnm{M\'onika} \sur{Polner}}\email{polner@math.u-szeged.hu}
\equalcont{These authors contributed equally to this work.}

\author[1]{\fnm{Daniel} \sur{Wylde}}

\author[2]{\fnm{Maria Vittoria} \sur{Barbarossa}}

\author[2,3]{\fnm{Gergely} \sur{R\"ost}}

\affil[1]{\orgdiv{School of Mathematics}, \orgname{University of Leeds}, \orgaddress{\street{Woodhouse}, \city{Leeds}, \postcode{LS2 9JT}, \state{United Kingdom}, \country{Country}}}


\affil[2]{\orgdiv{Bolyai Institute}, \orgname{University of Szeged}, \orgaddress{\street{Aradi v\'ertan\'uk tere 1}, \city{Szeged}, \postcode{H-6720}, \country{Hungary}}}

\affil[3]{\orgname{National Laboratory for Health Security}, \orgaddress{\street{Aradi v\'ertan\'uk tere 1}, \city{Szeged}, \postcode{H-6720}, \country{Hungary}}}


            

\abstract{
We investigate an epidemiological model that incorporates waning of immunity at the individual level and boosting of the immune system upon re-exposure to the pathogen. When immunity is fully restored upon boosting, the system can be expressed as an SIRS-type model with discrete and distributed delays.  
We conduct a numerical bifurcation analysis varying the boosting force and the maximum period of immunity (in the absence of boosting), while keeping other parameters fixed at values representative of a highly infectious disease like pertussis. 
The stability switches of the endemic equilibrium, identified numerically, are validated using an established analytical approach, confirming that the equilibrium is unstable in a bounded parameter region, and stable outside this region.
Using recently developed continuation methods for models with discrete and distributed delays, we explore periodic solutions and their bifurcations.
Our analysis significantly extends previous findings and reveals a rich dynamical landscape, including catastrophic bifurcations of limit cycles, torus bifurcations, and bistability regions where two stable periodic solutions coexist, differing in amplitude and period. 
These complex bifurcations have critical public health implications: perturbations—such as non-pharmaceutical interventions—can shift the system between attractors, leading to long-term consequences from short-term measures.
}




\maketitle

\section{Introduction}
Boosting of the immune system through re-exposure to an infectious pathogen plays a key role in the dynamics of diseases without lifelong immunity, affecting the success of vaccination programmes with imperfect or waning vaccine protection \citep{schuette1999modeling, heffernan2009implications, lavine2011natural, dafilis_frascoli_wood_mccaw_2012}. 

Waning and boosting of immunity are particularly challenging to describe mathematically as they introduce delayed feedback loops.
To keep the models tractable, the waning of immunity is often modeled using a chain of compartments with full or partial immunity. However, various analyses have shown that such models predict periodic outbreaks only if the chain is long enough, while an overly simplified model could disregard interesting dynamical outcomes and predict damped oscillations to an endemic equilibrium  \citep{hethcote1981nonlinear, gonccalves2011oscillations, rost2020stability}. For this reason, structured models that capture the within-host dynamics by explicitly keeping track of the time passed since infection or recovery are more appropriate than simple ordinary differential equations, allowing more flexibility in the definition of delay distributions and potentially reproducing more complex dynamics \citep{diekmann1982prelude, aron1983, barbarossa2015immuno, diekmann2018waning, kuniya, yuki}.

To describe in a general way the process of waning and boosting of immunity, \cite{barbarossa2015immuno} proposed a model where recovered individuals are structured by a simplified one-dimensional variable $z \in [z_{min},z_{max}]$ that describes the immunity level (thereby avoiding to model the complex details of the immune system response).
After natural infection, a host enters the immune compartment with maximum immunity~$z_{max}$, then their immunity decays in time with some given rate. Immunity can be boosted upon further exposure to the pathogen. When the immune level reaches the lowest threshold~$z_{min}$, the individual becomes susceptible again. 

In a later analysis, \cite{barbarossa2017stability} focused on the special case where immunity decays at a constant rate and boosting always restores the maximal level of immunity. For this reason, it is also referred to as the `MAXboost' model \citep{barbarossa_temporal_2018}. Under the assumption that the population size is constant, this model reduces to a system of two equations, for the susceptible and infected populations respectively, where the inflow into the susceptible compartment due to loss of immunity is mathematically captured by a distributed delay term. 
In this model, the maximal delay $\tau$ represents the maximal duration of immunity when boosting does not occur.
The authors showed the existence of an endemic equilibrium that undergoes multiple stability switches depending on the model parameters, with sequences of Hopf bifurcations suggesting the emergence of branches of periodic solutions. 

The linearized stability analysis of the endemic equilibrium for the MAXboost model was carried out numerically by \cite{barbarossa2017stability}. Here, we refine the analysis of stability switches of the endemic equilibrium when the maximum duration of immunity is varied, and fully characterize the stability changes that occur in the MAXboost model. The main difficulty lies in studying characteristic equations where the delay occurs in multiple ways. Delay systems in which the coefficients in the characteristic equation depend on the time delay $\tau$ only through an exponential term $\exp(-\tau d)$ are easier to study, and the theory in such cases is well developed (see, for example, \cite{Kuang1993}). One can usually explicitly compute the delay values for which stability switches occur. A systematic approach to studying characteristic equations with delay-dependent parameters has been developed by \cite{berettakuang2002} for one constant delay, and extended to multiple delay systems by \cite{Kuang_multiple_delays}. Using this method, we provide geometric and analytical criteria for stability switches with respect to the delay. In addition, this approach provides a rigorous method for finding Hopf bifurcations.

The bifurcation analysis beyond Hopf points for models with discrete and distributed delays has traditionally been limited, due to the complexity of the analytical calculations and the lack of easily accessible software tools. 
As an example, \cite{taylor2009sir} obtained analytical and numerical results on the delayed SIRS model without boosting of immunity by rigorously studying the small periodic oscillations emerging from a supercritical Hopf bifurcation, and the stable pulse oscillations that coexist with the stable endemic equilibrium near subcritical Hopf bifurcations. This is one of the first examples of bistability in infectious disease models with waning immunity. 

In the presence of boosting of immunity, bistability has also been observed by \cite{dafilis_frascoli_wood_mccaw_2012} in a SIRWS compartmental model with one intermediate waning immunity compartment. By performing a numerical bifurcation analysis, the authors showed the presence of subcritical Hopf bifurcations when varying the boosting force and the per capita mortality rate, with the emergence of regions of bistability of the endemic equilibrium with a stable limit cycle. \cite{kuhn} then rigorously showed existence of oscillations in the fast-slow SIRWS model proposed by \cite{dafilis_frascoli_wood_mccaw_2012} using geometric singular perturbation theory. \cite{childs} further extended the analyses of bistability in the SIRWS model, showing that the region of bistability is relatively small for biologically realistic parameters.
\cite{opoku-sarkodie_dynamics_2022} generalised the SIRWS model by allowing different expected durations for individuals being in the fully immune compartment and being in the waning immunity compartment, from where their immunity can still be restored upon re-exposure. The modified model exhibits rich dynamics and displays additional complexity with respect to the symmetric partitioning.

In this paper, we further expand the analysis of the MAXboost model by using a recently developed delay equation importer for MatCont~7p6 for MATLAB \citep{liessi2025matcont}, that allows to perform the numerical bifurcation analysis of equations with discrete and distributed delays. 
Our analysis captures a richer dynamical behaviour than previous ones, including those from \cite{hethcote1981nonlinear}, \cite{taylor2009sir} and \cite{dafilis_frascoli_wood_mccaw_2012}, even in the absence of boosting. In particular, we find a series of catastrophic bifurcations that originate from the switch between supercritical and subcritical Hopf bifurcations when varying the boosting force and the maximal duration of immunity, which can result in drastic changes in the epidemiological situation due to small perturbations of parameters. 
We find interesting bifurcations on the branches of periodic solutions, including Neimark--Sacker bifurcations and the emergence of stable tori, highlighting how waning and boosting of immunity lead to a very complex dynamical landscape. Most importantly, we identify regions of parameters that exhibit bistability, both of one equilibrum and a periodic solutions and of two distinct periodic solutions for the same parameter values. The existence of coexisting stable periodic solutions differing in amplitude and period has important consequences in public health, as perturbations due for instance to the implementation of public health interventions could push the system to a different attractor.

The paper is structured as follows. In Section~\ref{s:model}, we briefly summarise the MAXboost model under consideration and formally study the stability switches of the endemic equilibrium depending on the maximal duration of immunity when not boosted. 
In Section~\ref{s:num_bif}, we perform a numerical bifurcation analysis with respect to the length of the immune period and the boosting force. We also focus on the special case of waning immunity but no boosting, showing in this case, too, the existence of multiple stability switches and branches of stable periodic solutions with different amplitudes and periods. Finally, we briefly comment on the impact of the expected lifetime on stability.
In the last section, we discuss the results and the implications of our findings in the context of public health interventions. 


\section{The mathematical model}\label{s:model}
We consider the model studied by \cite{barbarossa2017stability} for the spread of an infectious pathogen that confers temporary immunity. 
Let $S$, $I$ and $R$ be the fraction of susceptible, infected, and immune individuals in the population, respectively, so that $S+I+R=1$. Let $d>0$ be the per capita natural mortality rate of individuals, equal to the total population birth rate. Let $\beta$ and $\gamma$ denote respectively the per capita transmission rate and the recovery rate of infected individuals. 
Upon recovery, individuals are completely immune to the disease, but they re-enter the susceptible class after time $\tau>0$, unless their immune system is boosted before time $\tau$. If immune individuals enter in contact with the pathogen, their immune system can be boosted to the maximal level at a rate $\beta\nu$, where $\nu \geq 0$ represents the boosting force. 
Hence, the rate at which individuals re-enter the class $S$ at time $t$ is given by the individuals who had maximal immunity at time $t-\tau$, and their immune system has not been boosted between $t-\tau$ and $t$. 
Specifically, at time $t-\tau$ there are two cohorts of individuals who acquire maximal immunity: those who recover naturally, at total rate $\gamma I(t-\tau)$, and those whose immunity has been boosted, at rate $\nu\beta I(t-\tau)R(t-\tau)$. The probability that individuals in these cohorts do not die nor receive an immunity boost between $t-\tau$ and $t$ is given by $e^{-d\tau -\nu \beta \int_{-\tau}^{0} I(t+u) \dd u}$. The total inflow into $S$ at time $t$ is therefore given by
\begin{equation*}
    \left[ \gamma I(t-\tau) + \nu\beta I(t-\tau)R(t-\tau)\right] \times \exp \left( -d\tau -\nu \beta \int_{-\tau}^{0} I(t+u) \dd u \right).
\end{equation*}
Since the population is constant, we can write $R=1-S-I$, hence the full model is described by the two equations
\begin{equation}
\begin{aligned}
\dot S(t) & = d(1-S(t)) -\beta I(t)S(t)\\
& \quad +I(t-\tau) \left(\gamma + \nu\beta \left(1-S(t-\tau)-I(t-\tau)\right)\right)
     \exp \left( -d\tau - \nu\beta\int_{t-\tau}^{t}I(u)\,\dd u\right)\\
     \dot I(t) & = \beta I(t)S(t) -(\gamma+d) I(t).
\end{aligned} \label{sys:SISdelay}
\end{equation}
We refer to \cite{barbarossa2015immuno} for an equivalent formulation of the model using partial differential equations. 

Note that, if $\nu\in[0,1]$, immune boosting following secondary exposure occurs at a lower rate than primary infection. Conversely, if $\nu\ge 1$, boosting can occur through exposures that would not induce infection in a susceptible host \citep{lavine2011natural}. 
When $\nu\to 0$, the model reduces to a SIRS model with no boosting of immunity. 
When $\nu \to +\infty$, the model reduces to a SIR model with no waning of immunity. 

System \eqref{sys:SISdelay} always admits the disease-free equilibrium $(S^*,I^*)=(1,0)$, which is globally asymptotically stable when $\r0<1$, where $\r0$ is the basic reproduction number, defined by
\begin{equation}\label{R0}
    \r0 = \frac{\beta}{\gamma+d}.
\end{equation}
If $\r0>1$, the disease-free equilibrium is unstable and a unique endemic equilibrium $(S^*,I^*)$ exists, with $S^* = \r0^{-1}$ and $I^*>0$ implicitly defined by the equation 
\begin{equation*}
    d(1-S^*) -\beta I^*S^*+\left(\gamma + \nu\beta \left(1-S^*-I^*\right)\right)
     I^* e^{ -d\tau - \nu\beta\tau I^*} = 0.
\end{equation*}
The endemic equilibrium can be stable or unstable depending on parameters.
\cite{barbarossa2017stability} performed a numerical analysis of the stability boundaries of the endemic equilibrium varying the parameters $\nu$ and $\tau$, with all other parameters fixed so that $\r0>1$, using TRACE-DDE, a MATLAB computational tool \citep{breda2009trace}. The analysis shows that the equilibrium destabilizes through Hopf bifurcations. However, no information is given about the dynamical behaviour when the endemic equilibrium is unstable.

\subsection{Stability switches of the endemic steady state}\label{sec:stability switches}
In this section, based on the method developed by \cite{berettakuang2002}, we study the occurrence of possible stability switches of the endemic equilibrium as a result of increasing the delay~$\tau$. 

Recall from \cite{barbarossa2017stability} that linearizing the system \eqref{sys:SISdelay} about the endemic equilibrium point $(S^*,I^*)$, with $I^*=I^*(\tau)\geq 0$, results in the characteristic equation
\begin{equation*}
    \begin{aligned}
    \lambda^3  & = -\lambda \beta(\gamma+d) I^* -\lambda^2 (d+ \beta I^*) - \lambda^2 e^{-\lambda \tau} \nu \mu\\
    & \quad + \lambda e^{-\lambda \tau} \mu \bigl( \sigma  - \nu\beta I^*\bigr)
    - \nu \beta I^* \sigma \mu+ e^{-\lambda \tau} \nu\beta I^*  \sigma \mu,
    \end{aligned}
\end{equation*}
where
\begin{equation}
    \begin{aligned} 
	\mu=\mu(\tau) & = \beta I^* e^{-\tau(d+ \nu\beta I^*)},\\
	\sigma=\sigma(\tau) &=\gamma + \nu\beta (1 -1/\r0-I^*).
    \end{aligned}\label{def:mu_sigma}
\end{equation}
This can be written in the form
\begin{equation*}
    W(\lambda, \tau)\equiv P(\lambda,\tau) +Q(\lambda,\tau) e^{-\lambda \tau}=0, 
\end{equation*}
where
\begin{align*}
    P(\lambda,\tau) & = \lambda^3 + \lambda^2 (d+ \beta I^*)+ \lambda \beta (\gamma+d)I^*+ \nu\beta I^*\mu\sigma,\\[0.2em]
    Q(\lambda, \tau) & =  \bigl(\lambda^2  \nu  - \lambda  \bigl( \sigma - \nu\beta I^*\bigr)- \nu \beta I^*  \sigma\bigr) \mu,
\end{align*}
and $\mu,\,\sigma$ as in \eqref{def:mu_sigma}. Note here the dependence of the polynomials $P$ and $Q$ on the delay, which appears in their coefficients not only explicitly in the term $e^{-\tau d}$, but also implicitly through the equilibrium component $I^*$. A systematic approach to study characteristic equations with delay-dependent parameters was developed by \cite{berettakuang2002}. We give here a brief summary of their technique as it applies to our problem. 

We look for purely imaginary roots, $\lambda = i\omega$, $\omega>0$ of $W(\lambda,\tau)$. Separating the real and imaginary parts in $W(i\omega,\tau)=0$ results in 
\begin{equation}\label{sin-cos-omega}
    \sin(\omega \tau) = \text{Im}\left(\frac{P(i\omega,\tau)}{Q(i\omega,\tau)} \right),\quad
    \cos(\omega \tau) = -\text{Re}\left(\frac{P(i\omega,\tau)}{Q(i\omega,\tau)} \right).
\end{equation}
From the characteristic equation, $P(i\omega,\tau)=-Q(i\omega,\tau)e^{-i\omega\tau}$ therefore, if $\omega$ satisfies~\eqref{sin-cos-omega}, then it must be a positive root of the polynomial 
\begin{equation}\label{eq:Fomega}
   F(\omega,\tau)\equiv|P(i\omega,\tau)|^2-|Q(i\omega,\tau)|^2 = \omega^2\left(\omega^4 +a_1(\tau) \omega^2 +a_0(\tau)\right),
\end{equation} 
where 
\begin{equation*}
    \begin{split}
    a_1(\tau) & =  d^2+ \beta^2 (I^*)^2 - 2\beta\gamma I^* -\nu^2  \mu^2,\\
    a_0(\tau) & =  (\beta I^*(\gamma+d))^2 -2\nu \beta I^*(d+ \beta I^*)\mu \sigma -(\sigma-\nu\beta I^*)^2\mu^2-2 \nu^2  \mu^2\beta I^*  \sigma.
\end{split}
\end{equation*}
Since the coefficients depend on the delay, we have that $\omega=\omega(\tau)$. We find that $\lambda=i\omega$ solves $W(i\omega,\tau)=0$ if and only if $\omega$ is a root of $F$, which in turn can have zero, one, or two positive $\omega(\tau)$ roots.

Assume that $J\subseteq \mathbb{R}^+$ is the set of all time delays for which $\omega(\tau)$ is a positive root of $F$, that is
\begin{equation*}
    J=\left\{\tau:\tau\geq0,\  \omega(\tau)\text{ is a positive simple root of } F(\omega,\tau)  \right\}.
\end{equation*}
For any $\tau\in J$, where $\omega(\tau)$ is a positive root of $F(\omega(\tau),\tau)$, we can define the angle $\theta(\tau)\in [0,2\pi]$, such that $\sin\theta(\tau)$ and $\cos\theta(\tau)$ are given by the right-hand side of \eqref{sin-cos-omega}, i.e.,
\begin{equation}\label{angles}
    \sin\theta(\tau) = \text{Im}\left(\frac{P(i\omega,\tau)}{Q(i\omega,\tau)} \right),\quad
    \cos\theta(\tau) = -\text{Re}\left(\frac{P(i\omega,\tau)}{Q(i\omega,\tau)} \right).
\end{equation}
Hence, for all $\tau\in J$ we must have $\omega(\tau)\tau=\theta(\tau)+2\pi n$. Define the maps $\tau_n: J\to \mathbb{R}^+$ as
\begin{equation*}
    \tau_n(\tau) = \frac{\theta(\tau)+2\pi n}{\omega(\tau)},\quad n\in\mathbb{N}_0,
\end{equation*}
where $\omega(\tau)$ is a positive solution of $F(\omega,\tau)$. Finally, introduce the functions 
\begin{equation*}
    S_n:J\to\mathbb{R},\quad S_n(\tau) = \tau-\tau_n(\tau), \quad n\in\mathbb{N}_0.
\end{equation*}
Then Theorem 2.2 in \cite{berettakuang2002} states that, if $\omega(\tau)$ is a positive root of $F(\omega,\tau)$ in \eqref{eq:Fomega} defined for $\tau\in J$, and at some $\tau^*\in\ J$ we have that 
\begin{equation*}
      S_n(\tau^*)=0 \quad \text{for some } n\in\mathbb{N}_0,
\end{equation*}
then a simple pair of conjugate pure imaginary roots $\lambda_\pm(\tau^*)=\pm i\omega(\tau^*)$ of $W(\lambda,\tau)=0$ exists at $\tau=\tau^*$ that crosses the imaginary axis. The theorem also gives conditions that determine the direction of the crossing. In particular, let 
    \begin{equation}\label{sign_crossing}
        \delta(\tau^*)=\text{sign } \left\{ \frac{\dd\text{Re}(\lambda)}{\dd\tau}\mid_{\lambda=i\omega(\tau^*)}\right\}= \text{sign }\left\{F'_\omega(\omega(\tau^*),\tau^*)\right\} \text{sign }\left\{ \frac{\dd S_n(\tau)}{\dd\tau}\mid_{\tau=\tau^*}\right\}.
    \end{equation}
 Then, the pair of eigenvalues crosses the imaginary axis from left to right if $\delta(\tau^*)>0$, and from right to left if $\delta(\tau^*)<0$.

Let us return to our model. We can give explicit conditions on the coefficients of the polynomial $F(\omega,\tau)$ in \eqref{eq:Fomega}, as functions of $\tau$, such that $F$ admits one or two positive $\omega(\tau)$ roots. In the next proposition, we characterize the set of time delays for which positive roots $\omega^\pm(\tau)$ exist. 

\begin{proposition}
    For the dynamical system \eqref{s:model}, the delay interval for which $F(\omega,\tau)$ has positive $\omega(\tau)$ roots is finite.
\end{proposition}
\begin{proof}
    First, we specify the delay intervals for which one or two positive (distinct) roots $\omega(\tau)$ of $F$ exist. It is straightforward to see that, if $a_0(\tau^*)<0$ for some $\tau^*>0$, then $F$ has exactly one $\omega^+(\tau^*)>0$ root, and, when $a_0(\tau^*)>0$, then under some additional conditions $F$ has exactly two $\omega^\pm(\tau^*)$ distinct positive roots. For a better understanding, we plotted in Figure~\ref{fig:roots_partition} all possible outcomes for the number of roots. This figure divides the $(a_0,a_1)$-plane according to the number of positive roots of $F$, and it also marks with a dashed line the boundary that separates regions where roots exist and where not. Along the dashed line, there are no admissible roots. Define the sets where positive roots of $F$ exist by
    \begin{equation*}
    J^- =\left\{\tau : \tau\geq0, \ a_0(\tau)>0 \ \text{and}\ D(\tau)>0  \ \text{and}\ a_1(\tau)<0\right\}
    \end{equation*}
and 
\begin{equation*}
    J^+ =J^- \cup \left\{\tau: \tau\geq0, \ a_0(\tau)<0\right\}\cup \left\{\tau: \tau\geq0, \ a_0(\tau)=0 \ \text{and}\ a_1(\tau)<0\right\},
\end{equation*}
where $D(\tau)=a_1^2(\tau)-4a_0(\tau)$. Hence for $\tau\not\in J^+\cup J^-$, $\omega(\tau)$ is not defined. It is clear that if $\tau\in J^+\setminus J^-$, then $F$ has only one $\omega^+(\tau)>0$ root, whereas if $\tau\in J^-$, then both $\omega^+(\tau)>0$ and $\omega^-(\tau)>0$ exist. Note that, due to the continuity of the roots with respect to $\tau$, $J^+$ and $J^-$ are (connected) intervals. For the remainder of this paper, let $J^+=(\tau^+_{min},\tau^+_{max})$ (or $J^+=[0,\tau^+_{max})$) and $J^-=(\tau^-_{min},\tau^-_{max})$ (or $J^-=[0,\tau^-_{max})$). We need to show that $\tau^\pm_{max}<\infty$. 
    \begin{figure}
        \centering
        \hspace*{-1cm}
        \includegraphics[scale=0.25]{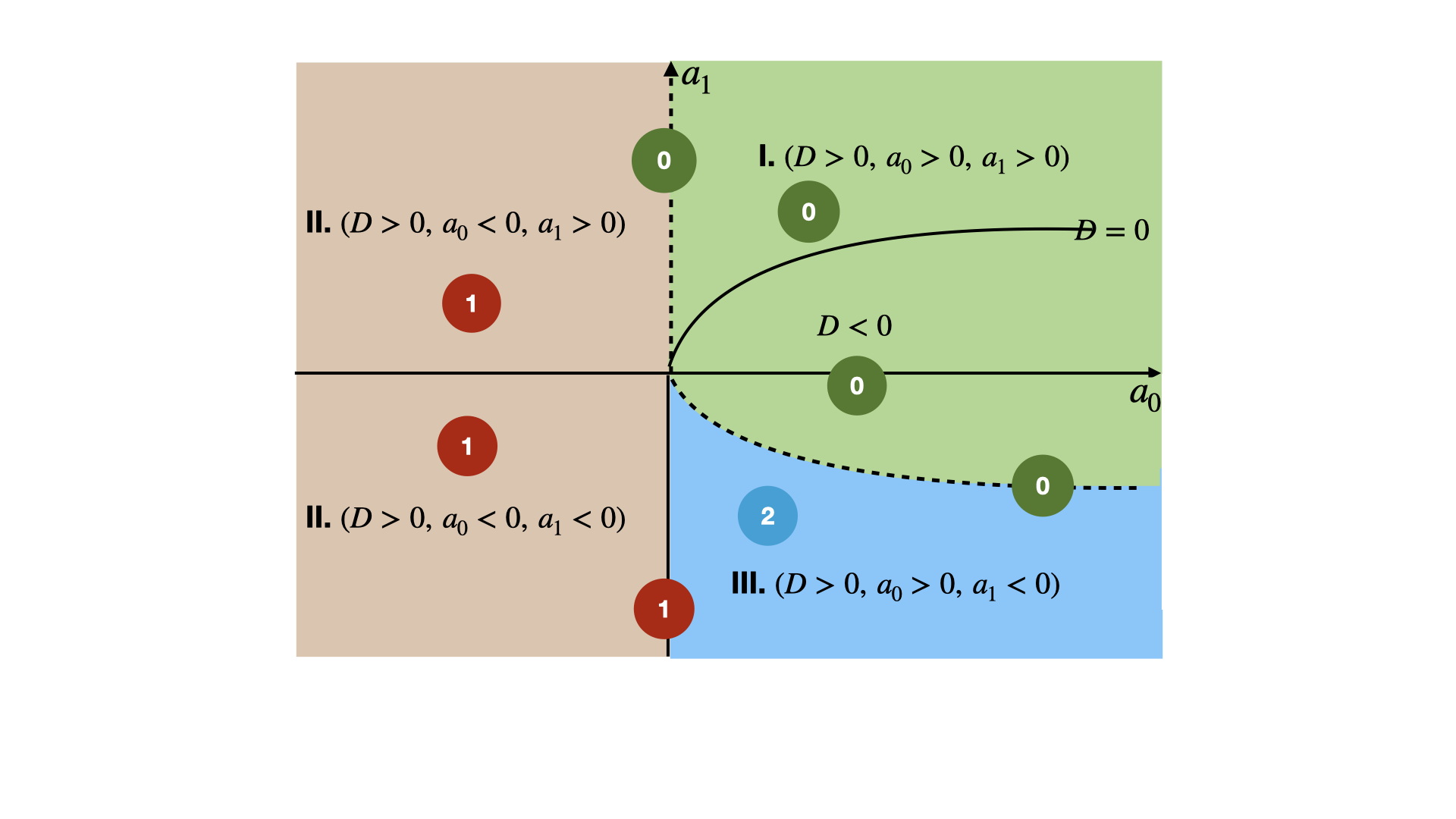}
        \caption{The partition of the $(a_0,a_1)$-plane according to the number of positive distinct roots $\omega(\tau)$ of $F(\omega,\tau)=0$ in \eqref{eq:Fomega}. The dashed curve is the boundary separating the region where there is at least one admissible root (marked by II. and III.) from the region where there are no roots (marked by I.).}
        \label{fig:roots_partition}
    \end{figure}

Consider the limit $\tau\to\infty$, which means that immunity is lifelong. In this case, the system reduces to an SIR system, and if there is demography ($d>0$), then there is a unique endemic equilibrium when $\r0>1$, 
\begin{equation*}
    S^*_\infty=\frac{1}{\r0},\quad I^*_\infty=\frac{d}{\beta}(\r0-1).
\end{equation*}
In this limit, we also have that 
\begin{equation}\label{a_D_infinity}
    \begin{split}    
    a_0^\infty &= \lim_{\tau\to\infty} a_0(\tau) = {d}^{2} \left( -\beta +\gamma +d \right) ^{2}>0, \\
    a_1^\infty &= \lim_{\tau\to\infty} a_1(\tau) = \frac{d}{(\gamma+d)^2}\left[ d\beta^2+2(\gamma+d)^2(\gamma+d-\beta)\right], \\
    D^\infty &= \lim_{\tau\to\infty} D(\tau) = \frac{d^3\beta^2}{(\gamma+d)^2}\left[ d\beta^2+4(\gamma+d)^2(\gamma+d-\beta)\right].
\end{split}
\end{equation}
Since $a_0^\infty>0$, there exists $\tau_1>0$ such that $a_0(\tau)>0$ for all $\tau>\tau_1$, i.e., if $\tau$ is large enough, we arrive and stay in the right half of the $(a_0,a_1)$-plane in Figure~\ref{fig:roots_partition}. Since the roots $\omega(\tau)$ are continuous, this implies that either there are no positive roots for $\tau>\tau_1$ (in domain I.), then $\tau_{max}^+=\tau_1$, or there are two distinct ones (in domain III.). In the second case, due to continuity, $a_1(\tau)<0$ for $\tau>\tau_1+T$ for some $T>0$ or $T=\infty$. 

If $T<\infty$, i.e., $a_1(\tau)$ changes sign at $\tau_1+T$, then again by continuity, this is only possible if we first cross the $D(\tau)=0$ curve at some $\tau$ value, i.e., there exists a $\tau_2<\tau_1+T$, such that $D(\tau_2)=0$. At this value of the delay, the two roots collide, yielding that $\tau_{max}^+=\tau_{max}^-=\tau_2$.

If $T=\infty$, i.e., $a_1(\tau)<0$ for all $\tau>\tau_1$, then in \eqref{a_D_infinity}, $a_1^\infty\leq0$. Since $\r0>1$, it follows that $D^\infty<0$, which implies due to the continuity of $D(\tau)$ that there exists $\tau_3>\tau_1$ such that $D(\tau_3)=0$. Therefore, $\tau_{max}^+=\tau_{max}^-=\tau_3$, which completes the proof.
    
\end{proof}

If only $\omega^+(\tau)$, $\tau\in J^+$ is feasible, then stability switches can occur only at the roots of $S_n^+(\tau)$. However, if both $\omega^+(\tau)$ and $\omega^-(\tau)$ are feasible for $\tau\in J=J^-$, then switches can occur at the zeros of the following two sequences of functions 
\begin{equation}\label{Sn_pm}
    S_n^+(\tau)=\tau-\frac{\theta^+(\tau)+2\pi n}{\omega^+(\tau)},\quad
    S_n^-(\tau)=\tau-\frac{\theta^-(\tau)+2\pi n}{\omega^-(\tau)},\quad n\in\mathbb{N}.
\end{equation}
Here the angles $\theta^+(\tau)$ and $\theta^-(\tau)$ are the solutions of \eqref{angles} corresponding to $\omega^+(\tau)$ and $\omega^-(\tau)$, respectively. Clearly, if $\omega^+(\tau)>\omega^-(\tau)$, then the following monotonicity properties hold for all $n\in\mathbb{N}_0$, $\tau\in J$,
\begin{equation}\label{monotonicity_S}
        S_n^+(\tau)>S_{n+1}^+(\tau),\ S_n^-(\tau)>S_{n+1}^-(\tau) \quad \text{and} \quad
        S_n^+(\tau)>S_{n}^-(\tau).
\end{equation}

\begin{remark}\label{remark_sign_f_derivative}
    Suppose that $\omega(\tau)$ is a root of $F(\omega(\tau),\tau)=0$ for $\tau\in J$. Since
    \begin{equation*}
        F'_\omega(\omega(\tau),\tau) = 2\omega(\tau) (2\omega^2(\tau)+a_1(\tau)),
    \end{equation*}
    it is straightforward to see that $F'_\omega(\omega^+(\tau),\tau) >0$ for all $\tau\in J^+$ and $F'_\omega(\omega^-(\tau),\tau) <0$ for all $\tau\in J^-$. Then we can determine the direction of which eigenvalues cross the imaginary axis, i.e., the sign of $\delta(\tau^*)$ in \eqref{sign_crossing}, only by analyzing the derivative of $S_n^+(\tau)$ and $S_n^-(\tau)$ at $\tau=\tau^*$, respectively.
\end{remark}
In the next lemma we describe, using Figure~\ref{fig:roots_partition}, how stability switches can occur, which will also determine the direction in which eigenvalues cross the imaginary axis. 

\begin{lemma}\label{lem:switches}
   The endemic equilibrium may undergo stability changes, and eventually it becomes stable for any value of the boosting $\nu$.      
\end{lemma}
\begin{proof}
We prove this lemma in more steps, depending on the number of feasible roots $\omega(\tau)$ of $F(\omega(\tau),\tau)$. 

\emph{Case 1.} Assume that $J^-=\emptyset$, i.e., only one root $\omega^+(\tau)$ is feasible for $\tau\in J^+=(\tau_{min},\tau_{max})$ (or $J^+=[0,\tau_{max})$). We need to show that for all $\tau$ larger than the last zero of $S_n^+(\tau)$ for some $n\in\mathbb{N}_0$, all eigenvalues have passed and remain on the left half of the imaginary axis. 

If $\tau_{min}=0$, then from \eqref{Sn_pm} it follows that $S_n^+(\tau_{min})<0$ for all $n\in\mathbb{N}_0$. If $\tau_{min}>0$, then it is determined by the conditions $a_0(\tau_{min})=0$ and $a_1(\tau_{min})>0$, see also Figure~\ref{fig:roots_partition}. Therefore,
\begin{equation*}
    \lim_{\tau\downarrow\tau_{min}}\omega^+(\tau)=0^+\text{ and } \lim_{\tau\downarrow\tau_{min}}S_n^+(\tau)=-\infty.
\end{equation*}
Since $J^+$ is finite, $\tau_{max}$ is also determined by the condition $a_0(\tau_{max})=0$ and $a_1(\tau_{max})>0$.  Then we have that 
\begin{equation*}
    \lim_{\tau\uparrow\tau_{max}}\omega^+(\tau)=0^+\text{ and } \lim_{\tau\uparrow\tau_{max}}S_n^+(\tau)=-\infty.
\end{equation*}
This, together with the monotonicity property of $S_n^+$ in \eqref{monotonicity_S}, implies that $S_n^+$ can only have an even number of sign changes for each $n$, so there is an even number of possible stability switches on $J^+$. Moreover, since the equilibrium is stable at $\tau=0$, the first switch is towards unstable. From Remark~\ref{remark_sign_f_derivative}, we have that $\delta(\tau^*)>0\ (<0)$ whenever $(S^+_n)'(\tau^*)>0\ (<0)$, so by the monotonicity property of $S_n^+$, it is straightforward to see that the last switch is at $\tau=\tau^*_{2m}$, the zero of $S_0^+(\tau)$, and that $\delta(\tau^*_{2m})<0$, which completes this part of the proof.   

\emph{Case 2.} Suppose $J^-\not=\emptyset$ and $\tau_{min}^+<\tau_{min}^-$. Then the switches in the interval $\tau\in(\tau_{min}^+,\tau_{min}^-)$ are determined by the zeros of $S_n^+(\tau)$. This means that as we increase $\tau$ in this interval, only $a_1(\tau)$ changes sign, which does not affect the existence of $\omega^+(\tau)$. At $\tau_{min}^-$ we have $a_0(\tau_{min}^-)=0$, $a_1(\tau_{min}^-)<0$ and $a_0(\tau)$ changes sign as we continue to increase $\tau$. Thus we enter region III in Figure~\ref{fig:roots_partition}, where both $\omega^+$ and $\omega^-$ exist up to $\tau_{max}^-$. Two things can happen here, either the two roots collide and then disappear, or only $\omega^+$ remains feasible as we further increase $\tau$. In the first case $\tau_{max}^+=\tau_{max}^-=\tau_{max}$, $D(\tau_{max})=0$ and $D(\tau)<0$ for $\tau>\tau_{max}$, and in the second case $a_0(\tau_{max}^-)=0$ and $a_1(\tau_{max}^-)<0$.

Let us analyze the possible stability switches that can occur at the zeros of $S_n^+(\tau)$, $\tau\in J^+$ and $S_n^-(\tau)$, $\tau\in J^-$, respectively. From case 1, it follows that the first switch occurs at $\tau_1^*$, the first zero of $S_0^+$ and the equilibrium becomes unstable. At $\tau_{min}^-$ the second root $\omega^-$ appears and 
\begin{equation*}
    \lim_{\tau\downarrow\tau_{min}^-}\omega^-(\tau)=0^+\text{ and } \lim_{\tau\downarrow\tau_{min}^- }S_n^-(\tau)=-\infty.
\end{equation*}
Next, we discuss the last eigenvalue crossing. If $\omega^+(\tau)$ and $\omega^-(\tau)$ collide (thus they are not admissible) at $\tau_{max}$, we have that
\begin{equation}\label{eq:limit_Sn_collide}
    \lim_{\tau\uparrow\tau_{max}} S_n^+(\tau)=\lim_{\tau\uparrow\tau_{max}} S_n^-(\tau)\quad\text{for all } n\in\mathbb{N}_0.
\end{equation}
If this limit is positive and $S_n^+$ and $S_n^-$ have sign changes, then the last zero at $\tau^*_m<\tau_{max}$ is at the zero of $S_n^-$. In this case, both $S_n^+$ and $S_n^-$ have an odd number of sign changes, so there are an even number of eigenvalue crossings in total. If the limit in \eqref{eq:limit_Sn_collide} is negative, then $S_n^+$ has an even number of sign changes and $S_n^-$ can have none or an even number for all $n>0$. Thus, the total number of zeros will also be even.

It remains to discuss the case where $\tau_{max}^-<\tau_{max}^+$, that is, only the root $\omega^+$ is feasible for $\tau>\tau_{max}^-$. On the interval $(\tau_{max}^-,\tau_{max}^+)$, we can apply the arguments of case 1 to count the number of sign changes of $S_n^+$. 

To summarize, in case 2 the number of eigenvalue crossings is even. Then Remark~\ref{remark_sign_f_derivative} and the monotonicity properties in \eqref{monotonicity_S} prove the statement of the lemma.

\emph{Case 3.} Suppose $J^+=J^-=(\tau_{min},\tau_{max})$ (or $[\tau_{min},\tau_{max})$), i.e., the roots $\omega^+(\tau)>\omega^-(\tau)>0$ are both feasible. If $\tau_{min}=0$, then $S_n^\pm
(0)<0$ for all $n\in\mathbb{N}_0$ and the monotonicity properties imply that the first switch occurs at the zero of $S_0^+(\tau)$. If $\tau_{min}>0$, then $D(\tau_{min})=0$ and $a_1(\tau_{min})<0$. Hence,
\begin{equation*}
    \lim_{\tau\downarrow\tau_{min}} S_n^+(\tau)=\lim_{\tau\downarrow\tau_{min}} S_n^-(\tau)\quad\text{for all } n\in\mathbb{N}_0,
\end{equation*}
and in the limit $\tau\uparrow\tau_{max}$ \eqref{eq:limit_Sn_collide} holds, so it is easy to check that the total number of eigenvalue crossings is also even in this case, and use the same arguments as before to complete the proof of this lemma.
\end{proof}
    
In the next section, we will use this technique to compute the stability switches for a given set of parameters inspired by pertussis, and different values of $\nu$.

\section{Numerical bifurcation analysis}\label{s:num_bif}
To investigate the dynamical behaviour of the system beyond the equilibria, in this section we perform a numerical bifurcation analysis for the parameter set considered by \cite{lavine2011natural} and \cite{barbarossa2017stability}, which is plausible with a highly infectious disease like pertussis. We take $\r0=15$, $\gamma=17$ (year$^{-1}$), corresponding to an average infectious period of 21 days, and per capita death rate $d=0.02$ (year$^{-1}$), corresponding to an average lifetime of 50 years. The parameter $\beta$ is then fixed through~\eqref{R0} ($\beta=255.3$ year$^{-1}$). 
We perform stability and bifurcation analyses varying the parameters $\tau$ and $\nu$, which represent the maximal duration of immunity in the absence of boosting and the boosting force, respectively. 
Since $\r0>1$, the endemic equilibrium exists for all $\nu$ and $\tau$, with the susceptible population equal to $S^*=\r0^{-1}=\frac{1}{15}$ at equilibrium. 

The numerical continuation is performed with MatCont~7p6 (running on MATLAB). For the approximation of the history in the delay interval we chose a collocation degree of $M=20$ or $M=30$, depending on the required accuracy and the other parameter values. The distributed delay is approximated with a quadrature rule on the same collocation nodes. 
We refer to \cite{breda2016pseudospectral} for further details on the numerical approximation. 

\begin{figure}[ht]
    \centering
    \includegraphics[scale=.9]{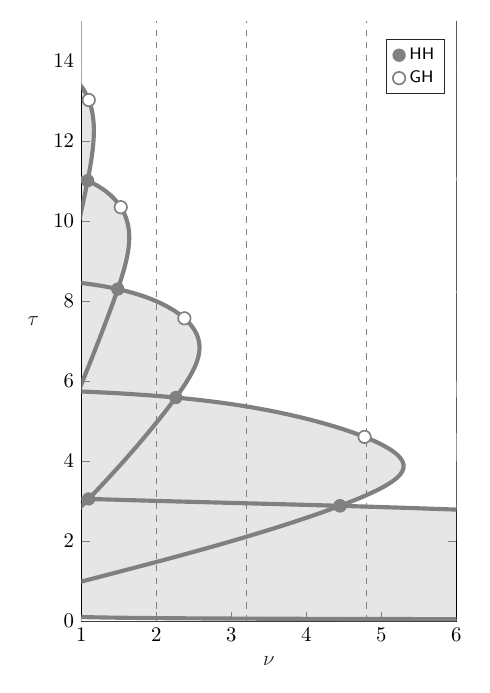}
    \includegraphics[scale=.9]{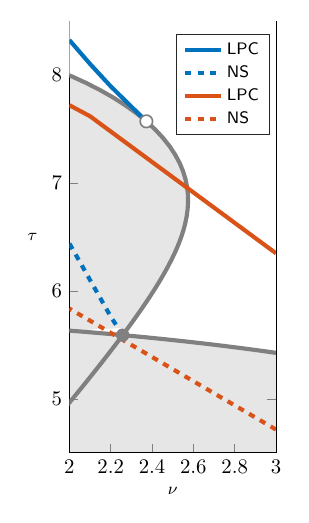}
    \caption{Left: stability of the endemic equilibrium of model \eqref{sys:SISdelay} in the plane $(\nu,\tau)$. Solid lines represent Hopf bifurcation curves, with Hopf-Hopf (HH, $\bullet$) and Generalized Hopf (GH, $\circ$) bifurcations detected along the curves. The endemic equilibrium is unstable in the shaded area. Although not visible in the diagram due to the scaling of the axis, the lowest curves connect at a large value of $\nu$. The vertical dashed lines correspond to the parameters chosen for the one-parameter bifurcation diagrams in the following analyses. 
    Note that an accurate approximation near $\nu=0$ requires large discretization indices and becomes computationally challenging, so we do not show it here. Right: close-up with the bifurcations numerically detected on different branches of limit cycles, delimiting their stability regions.
    }
    \label{fig:stability}
\end{figure}

\paragraph{Stability of the endemic equilibrium}
Figure~\ref{fig:stability} shows the computed stability regions of the endemic equilibrium in the parameter plane $(\nu,\tau)$, which is consistent with the results presented by \citet[Figure 1]{barbarossa2017stability}. 
The endemic equilibrium is stable outside the outermost boundary curve, and unstable inside (shaded region). Note that, as theoretically expected, the equilibrium is always stable on the axis $\tau=0$ (corresponding to an SIS-type model with individuals immediately susceptible upon recovery). 

The boundary separating the stable and unstable regions is the union of different segments of Hopf bifurcation curves corresponding to different pairs of complex eigenvalues crossing transversally the imaginary axis. 
The parameter values at which two different pairs of complex eigenvalues cross the imaginary axis correspond to Hopf-Hopf bifurcations (HH), and lie at the intersection of distinct Hopf curves. 
Moreover, generalized Hopf points (GH) are detected along each curve, corresponding to changes in the criticality of the Hopf bifurcation. 
We refer to \citet[Chapter 8]{kuznetsov1998elements} for further details on these bifurcations. 

GH points are often associated with the emergence of limit points of cycles (LPC) \cite[Section 8.3.3]{kuznetsov1998elements} and hence with bistability of equilibria and periodic solutions. Similarly, HH points are associated with the emergence of Neimark--Sacker bifurcations (NS) and invariant tori \cite[Section 8.6.3]{kuznetsov1998elements}. Figure \ref{fig:stability} therefore suggests the emergence of rich and complex dynamics. We further explore such complex dynamics by studying the one-parameter continuation of the periodic orbits emerging from Hopf and their bifurcations, fixing $\nu = 4.8$, $3.2$, $2$, $1$, and varying the duration of immunity $\tau$ (along the dashed lines indicated in Figure~\ref{fig:stability}). 

\begin{figure}[t]
    \centering
    \includegraphics[scale=.9]{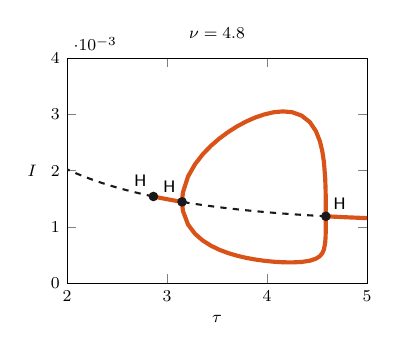}
    \includegraphics[scale=.9]{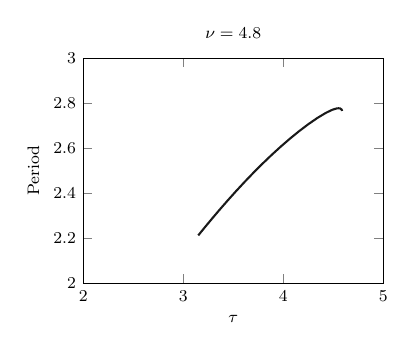}
    \caption{Left: bifurcation diagram of $I$ varying $\tau$, for $\nu=4.8$, including the equilibrium curve with the detected Hopf bifurcations (H, $\bullet$) and the max/min values of the branch of limit cycles. Solid red lines indicate a stable equilibrium/limit cycle, dashed lines are unstable. Right: period of the limit cycle branch (years).}
    \label{fig:4.8}
\end{figure}

\begin{figure}[t]
    \centering
    \includegraphics[scale=.9]{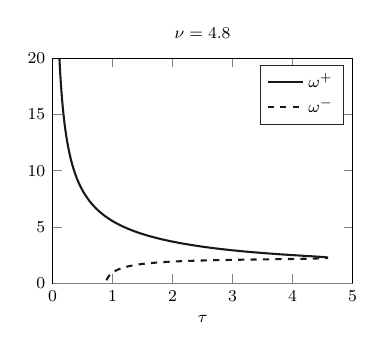}
    \includegraphics[scale=.9]{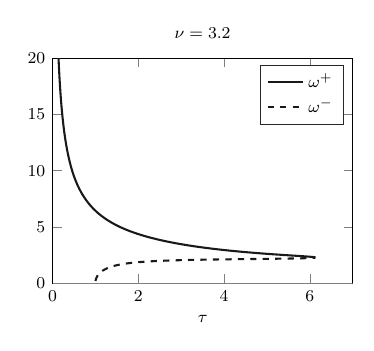}
    \caption{Positive roots $\omega^+(\tau)$ and $\omega^-(\tau)$ of $F(\omega,\tau)$ in \eqref{eq:Fomega} for $\nu=4.8$ (left) and $\nu =3.2$ (right) as $\tau$ varies. }
    \label{fig:omega_pm_nu4.8_3.2}
\end{figure}
 
\paragraph{One-parameter bifurcation diagrams with respect to $\tau$}
Figure \ref{fig:4.8} shows the one-parameter bifurcation diagram with respect to $\tau$ for $\nu=4.8$ (left), and the corresponding estimated period along the branch of limit cycles (right). Four Hopf bifurcations are detected along the equilibrium curve, at approximately $\tau \approx 0.05$, $2.86$, $3.15$, and $4.59$. The equilibrium is unstable between the first two points and between the last two points. A bubble of stable limit cycles exists for $\tau$ between $3.15$ and $4.59$ (both supercritical bifurcations), with a period between 2.2 and 2.8 years. 
When computing the limit cycle branch emerging from $\tau\approx0.05$, the amplitude of the branch diverges quickly, and the numerical continuation stops. This branch is therefore not shown in the figure. 

Figure~\ref{fig:omega_pm_nu4.8_3.2} (left) shows how the roots $\omega^+(\tau)$ and $\omega^-(\tau)$ of $F(\omega,\tau)$ in \eqref{eq:Fomega} vary as a function of $\tau$ in their maximal interval of existence, which for this set of parameters are $J^+=[0,4.59)$ and $J^-=(0.9,4.59)$, respectively. Let us analyze this with reference to Figure~\ref{fig:roots_partition}. When $\tau=0$, we are in the second quadrant of the $(a_0,a_1)$-plane, so $\omega^+(0)>0$ is the only root. As we increase $\tau$, only $a_1(\tau)$ changes sign, which does not affect the existence of $\omega^+$. At $\tau_{min}^-=0.9$, $a_0(\tau_{min}^-)=0$, $a_1(\tau_{min}^-)<0$ and $a_0$ changes sign as we continue to increase $\tau$. Thus we enter region III, so both $\omega^+$ and $\omega^-$ exist up to $\tau_{max}=4.59$, where they collide and then disappear. This is an example of case~2 in the proof of Lemma~\ref{lem:switches}.

Let us analyze the possible stability switches that can occur at the zeros of $S_n^+(\tau)$, $\tau\in J^+$ and $S_n^-(\tau)$, $\tau\in J^-$, respectively. Since $S_0^+(0)<0$, the first switch occurs at $\tau=\tau^*_1$, the first zero of $S_0^+$, destabilizing the endemic equilibrium. Since at $\tau_{max}=4.59$, $\omega^+(\tau)$ and $\omega^-(\tau)$ collide and disappear, we have that
\begin{equation*}
    \lim_{\tau\uparrow\tau_{max}} S_n^+(\tau)=\lim_{\tau\uparrow\tau_{max}} S_n^-(\tau)\quad\text{for all } n\in\mathbb{N}_0.
\end{equation*}
This limit is positive for $n=0,1$, for which $S_n^+$ and $S_n^-$ have a sign change, so both $S_n^+$ and $S_n^-$ have an odd number of zeros, and therefore together they have an even number. The last zero is at $\tau^*_m<\tau_{max}$, the zero of $S_1^-(\tau)$.  

As we can observe in Figure~\ref{fig:eigenvalues}, the numerical computation of the rightmost eigenvalue crossings as $\tau$ varies (bottom left panel), aligns perfectly with the computation of possible stability switches, which are the zeros of $S^\pm_n(\tau)$, $n=0,1$ (top left panel). The zeros were calculated using \eqref{Sn_pm} on the corresponding intervals $J^+$ and $J^-$. This also confirms Lemma~\ref{lem:switches}, which implies that the endemic equilibrium will be stable for all $\tau>\tau_{max}$.

\begin{figure}[tp]
    \centering
    \includegraphics[scale=.9]{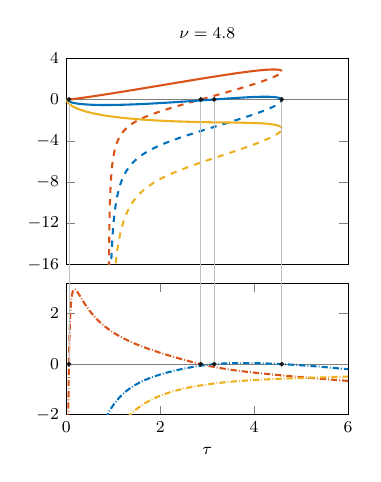}
    \includegraphics[scale=.9]{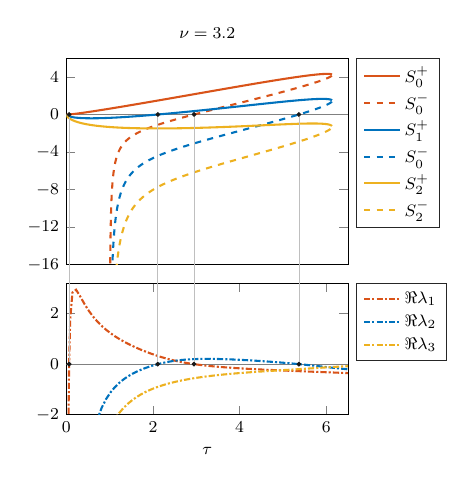}
    \caption{Top: graph of the stability switch functions $S_n^\pm (\tau)$ for $n=0,1,2$, where $\nu=4.8$ and $3.2$, respectively. Bottom: corresponding real part of the rightmost eigenvalues along the equilibrium branch, varying $\tau$. Black dots indicate Hopf bifurcations.}
    \label{fig:eigenvalues}
\end{figure}

To illustrate the emergence of bistability through the GH point, Figure~\ref{fig:3.2} shows the one-parameter continuation for $\nu=3.2$. In this case, we still observe a bubble of limit cycles for $\tau$ between the Hopf points~$2.11$ and $5.37$ (first supercritical, second subcritical). An LPC and an NS bifurcations are detected on the branch of limit cycles, and the limit cycles are stable between those bifurcations. 
Bistability occurs for values of $\tau$ between the Hopf and LPC bifurcations (i.e., for $\tau$ between $5.37$ and $6.16$), where a stable limit cycle coexists with the stable endemic equilibrium, so the asymptotic behaviour of the system depends on the initial condition. 

Note that, once again, the Hopf bifurcations correspond to those found analytically: in Figure~\ref{fig:omega_pm_nu4.8_3.2} (right) the roots $\omega^+(\tau)$ and $\omega^-(\tau)$ are plotted as functions of $\tau$ on their maximal interval of existence, i.e., $J^+=[0,6.13]$, $J^-=(1,6.13]$. The stability switches and rightmost eigenvalue crossings can be observed in the right panel of Figure~\ref{fig:eigenvalues}.

\begin{figure}[t]
    \centering
    \includegraphics[scale=.9]{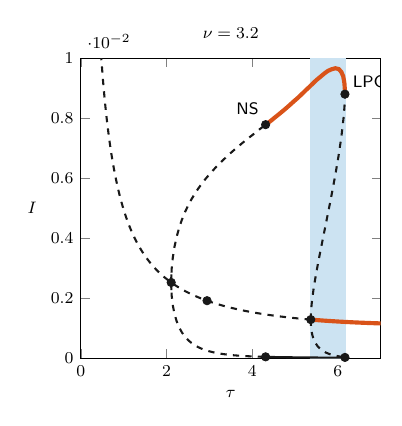}
    \includegraphics[scale=.9]{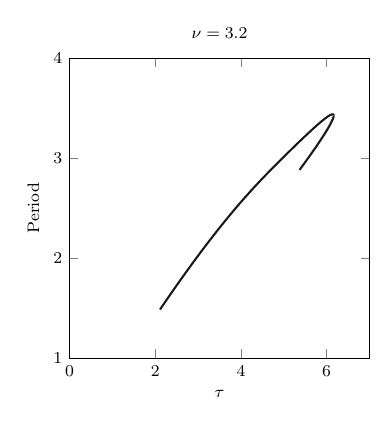}
    \caption{Left: bifurcation diagram of $I$ varying $\tau$, for $\nu=3.2$. Solid red lines indicate stable equilibrium/limit cycles; dashed lines indicate unstable. NS and LPC labels indicate Neimark-Sacker and limit point of cycles bifurcations, respectively. The blue shading shows the interval in which a stable equilibrium coexists with a stable limit cycle, for $\tau$ between $5.37$ and $6.16$. 
    Right: period along the limit cycle branch (years). 
    }
    \label{fig:3.2}
\end{figure}
\begin{figure}[t]
    \centering
    \includegraphics[scale=.9]{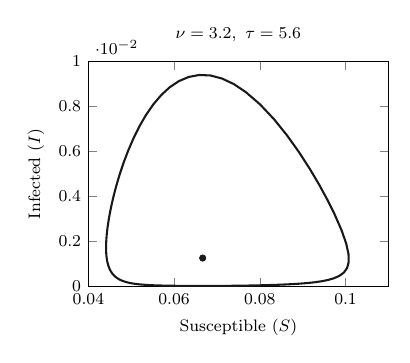}
    \includegraphics[scale=.9]{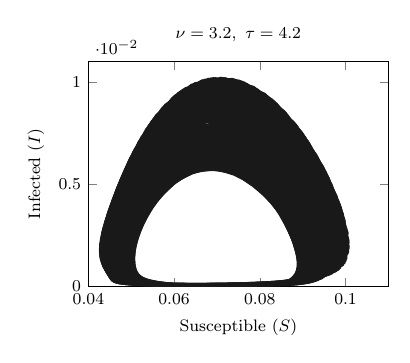}
    \caption{Left: bistability of the positive equilibrium and the stable limit cycle for $\nu=3.2$ and $\tau=5.6$. Right: emergence of a stable torus via a Neimark--Sacker bifurcation, for $\nu=3.2$ and $\tau=4.2$.}
    \label{fig:3.2_NS}
\end{figure}

Figure~\ref{fig:3.2_NS} shows the emergence of a stable invariant torus through the NS bifurcation: the long-term orbits in the plane $(S,I)$ are plotted for values of $\tau$ taken on opposite sides of the NS ($\tau=5.6$ and $\tau=4.2$). The left panel shows the coexisting stable equilibrium and limit cycle, while the right panel shows the orbit of a stable invariant torus.

An even more complex dynamical picture is shown in Figure \ref{fig:2} for $\nu=2$, with two different closed limit cycle curves connecting different pairs of Hopf bifurcations, one with larger amplitude and larger period than the other. 
Each branch is stable in an interval of values between an LPC and an NS bifurcation. 
Two regions of bistability appear: one with a stable equilibrium coexisting with a stable limit cycle, and one with two different stable limit cycles. These regions are shaded in Figure \ref{fig:2}. 
Figure \ref{fig:2_trajectory} shows the orbits of the limit cycles in the small- and large-amplitude branches, as well as their time profiles. The figure highlights the differences in both amplitude and period. 
Following the LPC and NS curves for each bubble of limit cycles in the plane $(\nu,\tau)$ one can get some insights about regions of parameters where two stable periodic solutions coexist, or one stable periodic solution coexists with the stable endemic equilibrium. These curves are sketched in the right panel in Figure \ref{fig:stability}.

\begin{figure}[tp]
    \centering
    \includegraphics[scale=.9]{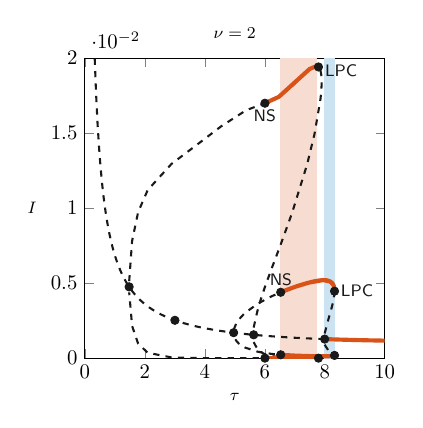}
    \includegraphics[scale=.9]{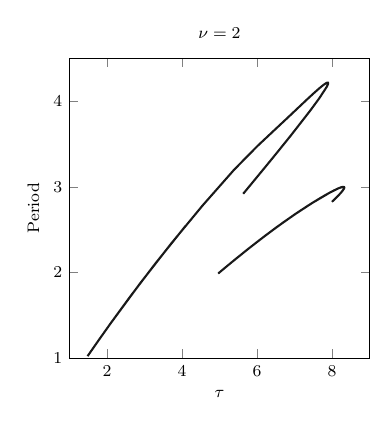}
    \caption{Left: bifurcation diagram of $I$ varying $\tau$, for $\nu=2$. Solid red lines indicate stable equilibrium/limit cycles; dashed indicate unstable. 
    Right: period of the limit cycle along the branches.
    The figure shows two intervals of bistability: the blue shading indicates bistability of an equilibrium and limit cycle, while the red shading indicates bistability of two limit cycles differing in amplitude and period. 
    }
    \label{fig:2}
\end{figure}

\begin{figure}[tp]
    \includegraphics[scale=.9]{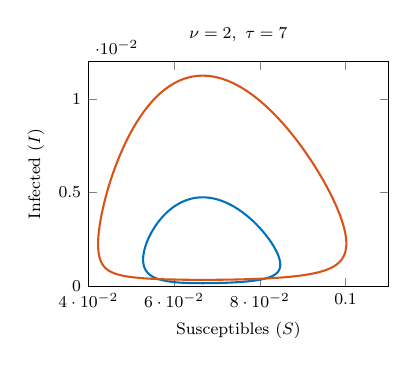}
    \includegraphics[scale=.9]{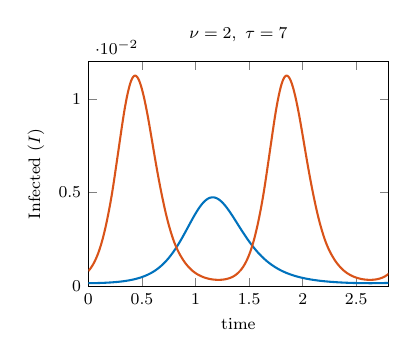}
    \caption{Orbits (left) and trajectories (right) of the stable coexisting limit cycles, for $\nu=2$ and $\tau=7$. Periods are approximately $1.41$ years (blue) and $2.67$ years (red).
    }
    \label{fig:2_trajectory}
\end{figure}

Figure~\ref{fig:omega_pm_nu2} shows the $\omega^+$ and $\omega^-$ curves expected from the analysis with $\nu=2$, as well as the stability switch curves that predict the Hopf bifurcation points. 

\begin{figure}
    \centering
    \includegraphics[scale = 0.9]{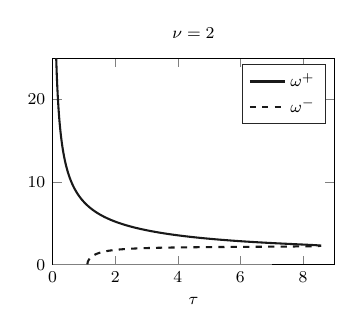}
    \includegraphics[scale = 0.9]{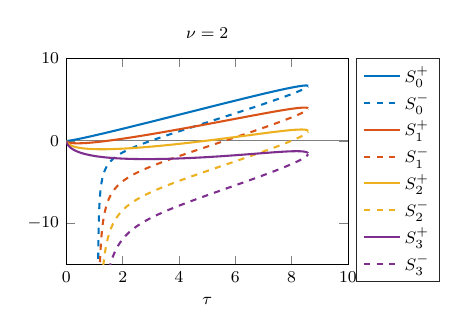}
    \caption{For $\nu=2$, the positive roots $\omega^+(\tau)$ and $\omega^-(\tau)$ of $F(\omega,\tau)$ in \eqref{eq:Fomega} (left) and the graph of the stability switch functions $S_n^\pm (\tau)$ for $n=0,1,2,3$ (right). 
    }
    \label{fig:omega_pm_nu2}
\end{figure}

\begin{figure}[tp]
    \includegraphics[scale=.9]{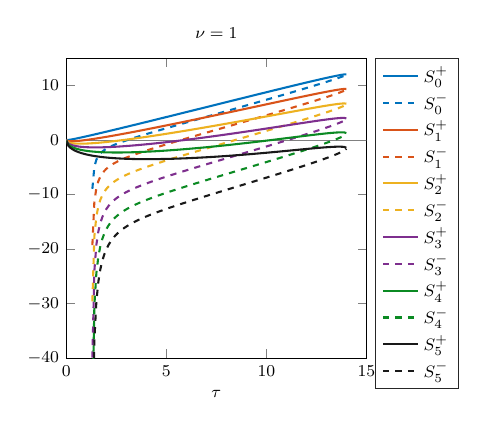}
    \includegraphics[scale=.9]{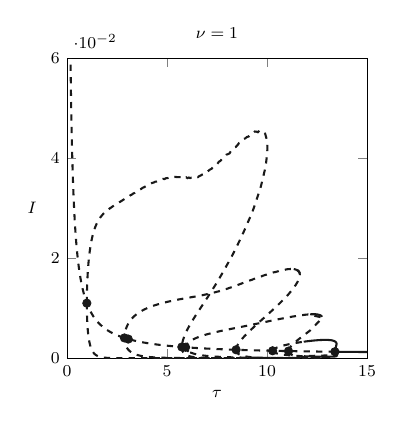}
    \caption{Case $\nu=1$. Left: stability switch functions $S_n^\pm(\tau)$. Right: bifurcation diagram, with dashed lines indicating unstable branches. 
    }
    \label{fig:1}
\end{figure}

As $\nu$ becomes smaller, the bifurcation diagrams become more and more complicated, with several nested limit cycle bubbles emerging from longer sequences of Hopf bifurcations. Figure \ref{fig:1} shows the stability switch curves for $\nu=1$, and the corresponding bifurcation diagram. In the left panel, we can see five different curves crossing the horizontal axis, giving rise to ten Hopf bifurcations. The last bifurcation, which corresponds to a stability switch from unstable to stable, is at approximately $\tau\approx 13.39$. The four rightmost bubbles of limit cycles are shown in the right panel. As before, the branch of limit cycles emerging from the smallest Hopf point is not shown, as it quickly diverges and the computation stops.

\begin{figure}[tp]
    \centering
    \includegraphics[scale=0.9]{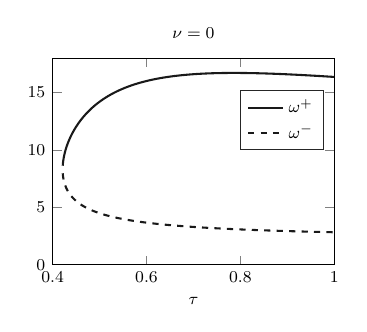}
     \includegraphics[scale=0.9]{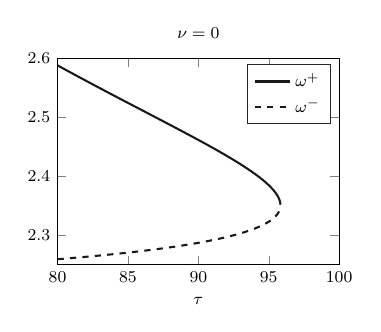}
    \\
    \includegraphics[scale=0.9]{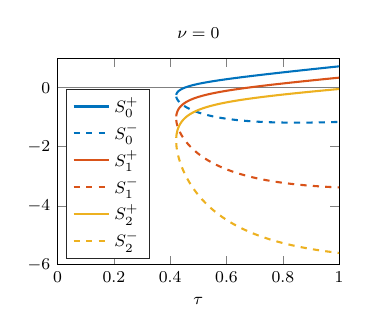}
    \includegraphics[scale=0.9]{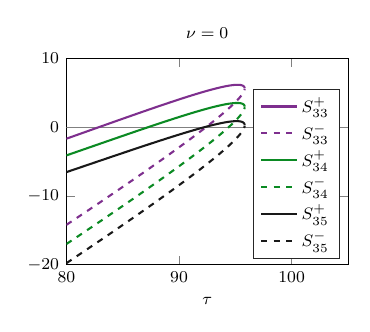}
    \caption{Case $\nu=0$. Top row: the roots $\omega^-(\tau)$ and $\omega^+(\tau)$ of $F(\omega(\tau),\tau)$ as functions of $\tau$ when $\tau$ is small, close to $\tau_{min}$, and large, close to $\tau_{max}$. Bottom row: the stability switches in the corresponding intervals. The last switch, when the equilibrium becomes stable, is the zero of $S_{35}^-$ at $\tau\approx 95.805$.}\label{fig:Sn_nu0}
\end{figure}

\begin{figure}[tp]
\centering
    \includegraphics[scale=.9]{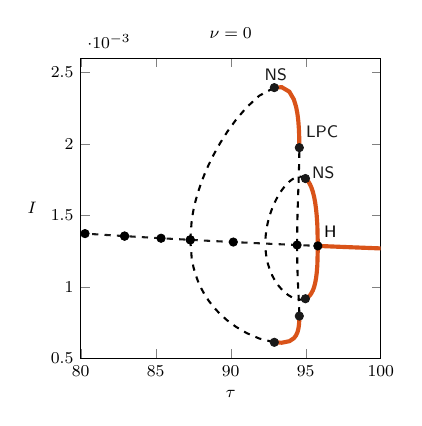}
    \includegraphics[scale=.9]{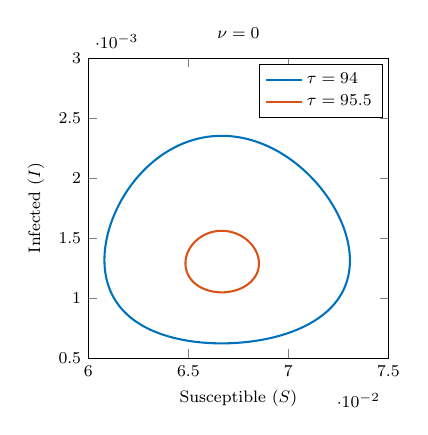}
    \caption{
    Case $\nu=0$. Left: bifurcation diagram of $I$ (rightmost Hopf bifurcations) obtained with DDE-BIFTOOL, showing two regions of stability on different branches of periodic orbits. Right: stable limit cycles with different amplitudes and periods, obtained via long-term time integration using \texttt{dde23}. 
    }
    \label{fig:nu0_bif}
\end{figure}

\subsection{The case with no boosting}
When $\nu=0$ (no boosting) the system reduces to the delayed SIRS model
\begin{align*}
    \dot S(t) &= d(1-S(t))-\beta I(t)S(t)
        +\gamma e^{-d\tau}I(t-\tau) \\
    \dot I(t) &= \beta I(t)S(t)-(\gamma + d)I(t).
\end{align*}
Similar models with a fixed immunity period and no boosting of immunity have been considered for instance by \cite{hethcote1981nonlinear}, \cite{taylor2009sir}, and more recently applied to COVID-19 transmission by \cite{pell2023emergence}.

Since the unique endemic equilibrium 
\begin{align*}
    I^* &= \frac{d(\beta - \gamma - d)}{\beta (\gamma+d-\gamma e^{-d\tau})},  \quad S^* = \frac{\gamma+d}{\beta}
\end{align*}
depends on $\tau$ explicitly, the calculation of stability switches in Section~\ref{sec:stability switches} becomes somewhat easier in this case. We seek zeros $\omega^+(\tau)>\omega^-(\tau)>0$ of $F(\omega,\tau)$ in \eqref{eq:Fomega}, where the coefficients are 
\begin{align*}
    a_1(\tau) & =  d^2+ (\beta I^*)^2 - 2\gamma \beta I^*,\\
    a_0(\tau) & = (\beta I^*)^2 (\gamma+d-\gamma e^{-\tau d})(\gamma+d+\gamma e^{-\tau d}).
\end{align*}
Since $a_0(\tau)>0$ for all $\tau$ and all parameters in the system, $F$ has two or no positive zeros. Hence, $\omega^+(\tau)$ and $\omega^-(\tau)$ are both feasible when $\tau\in J^+ = J^- = J$, where 
\begin{equation*}
    J=\left\{\tau : \tau>0 : a_1^2(\tau)-4a_0(\tau)>0 \text{ and } a_1(\tau)<0\right\}.
\end{equation*}
This corresponds to case 3 in the proof of Lemma~\ref{lem:switches}. With our set of parameters, $J=(0.421,95.83)$. Since this interval is very large, we have covered only the region close to $\tau_{min}$ and $\tau_{max}$, where $\omega^\pm$ are feasible and collide at both ends, see Figure~\ref{fig:Sn_nu0} (top panels). There are $72$ possible stability switches, so in Figure~\ref{fig:Sn_nu0} (bottom panels) we show only the first and last few $S_n^\pm(\tau)$ curves. The last switch is at the zero of $S_{35}^-$ at $\tau\approx 95.80$, Figure~\ref{fig:Sn_nu0} (bottom right).

Since the model has only one discrete delay, we performed the bifurcation analysis with the package DDE-BIFTOOL for MATLAB \citep{ddebiftool}. Figure \ref{fig:nu0_bif} (left) shows the rightmost Hopf bifurcations and the two rightmost limit cycle bubbles. In this case, we did not detect bistability, but two stable limit cycles with different amplitudes exist for parameter values that are in a relatively small interval around $\tau=95$ years.
The right panel in Figure \ref{fig:nu0_bif} shows the stable orbits in the $(S,I)$ plane, for $\tau=94$ (large amplitude) and $\tau=95.5$ (small amplitude). The convergence to the stable limit cycles was confirmed by performing a long-time integration of the system using the built-in function \texttt{dde23} for MATLAB. 

Finally, we stress that, due to the large values of the delay at the stability switch when $\nu=0$, the numerical analysis with MatCont becomes challenging as large discretization indices are needed to obtain sufficient accuracy. 
Figure \ref{fig:error} compares the approximation error of the rightmost Hopf bifurcation detected by MatCont, for $\nu=3.2$ and $\nu=0$, depending on the degree $M$ of the collocation polynomial used in the approximation of the delay system in MatCont~7p6. The right panel shows that a collocation degree of the order of $150$ is needed to approximate the Hopf bifurcation for $\nu=0$, which occurs at approximately $\tau\approx95.80$. This is much larger than the collocation degree required to approximate the Hopf bifurcations for larger $\nu$ (left panel), which is of the order of $20$. 

\begin{figure}[p]
    \centering
    \includegraphics[scale=.9]{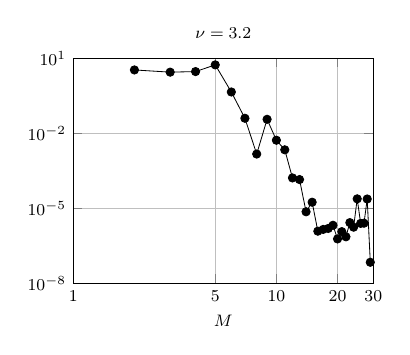}
    \includegraphics[scale=.9]{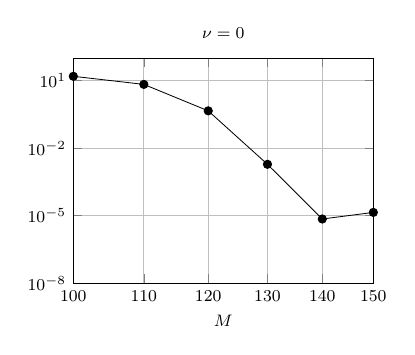}
    \caption{Log-log plot of the approximation error of the rightmost Hopf point for $\nu=3.2$ (left) and $\nu=0$ (right), depending on the collocation degree $M$ used in the approximation of the delayed state in MatCont~7p6. 
    The error is computed with respect to the numerical value obtained with $M=31$ and $M=160$, respectively. The computation is carried out using the default MatCont tolerance $10^{-6}$, explaining the error barrier. 
    }
    \label{fig:error}
\end{figure}

\begin{figure}[p]
    \centering
    \includegraphics[scale=.7]{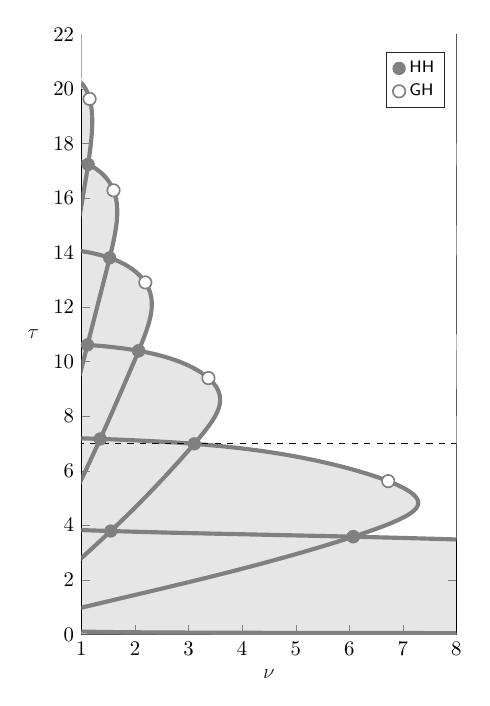}
    \includegraphics[scale=.7]{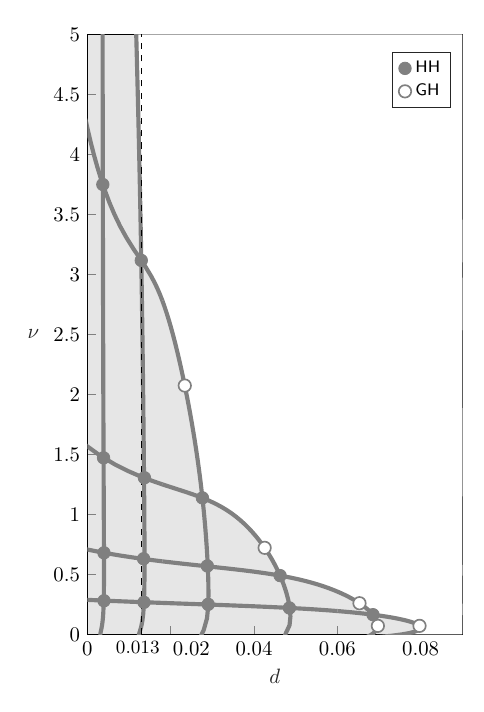}
    \caption{
    Left: stability of the endemic equilibrium of model \eqref{sys:SISdelay} in the plane $(\nu,\tau)$, with same parameters as in Figure~\ref{fig:stability}, but the per capita mortality rate now fixed at $d=0.013$, corresponding to an expected lifespan of approximately 75 years. Right: stability of the endemic equilibrium in the plane $(d,\nu)$, for fixed $\tau=7$.
    }
    \label{fig:d0013}
\end{figure}

\subsection{Impact of life expectancy}

The previous analyses were carried out for fixed demographic parameters, and in particular for a rather short life expectancy (50 years). However, life expectancy is known to play an important role in the emergence of oscillatory dynamics: using a compartmental SIRWS model, \cite{dafilis_frascoli_wood_mccaw_2012} showed that increasing life expectancy can cause sustained oscillations in the presence of waning and boosting of immunity. 
With this motivation, we conclude this study with a brief discussion on the sensitivity of our stability results to the life expectancy of individuals.

The left panel of Figure~\ref{fig:d0013} shows the stability and instability regions of the endemic equilibrium in the parameter plane $(\tau,\nu)$ (similarly as Figure~\ref{fig:stability}) when the per capita death rate is $d=0.013$ years$^{-1}$, corresponding to an average lifespan of approximately $75$ years. 
While the qualitative picture remains the same, decreasing the mortality rate $d$ enlarges the region of instability compared to Figure \ref{fig:stability} (left). The destabilizing effect of increasing life expectancy is also evident from the right panel of Figure \ref{fig:d0013}, which shows the stability and instability regions in the plane $(d,\nu)$, for fixed $\tau=7$ years and $\beta = 255.3$ years$^{-1}$, so that $\mathcal{R}_0$ varies with $d$ according to \eqref{R0}. We observe that, as $d$ decreases, the system undergoes a Hopf bifurcation with the emergence of stable periodic solutions. 
These results qualitatively agree with the analyses by \cite{dafilis_frascoli_wood_mccaw_2012} (see for instance Figure 4 in that reference).

\section{Discussion and conclusions}
In this paper, we extended the analysis of the dynamics and bifurcations of the model with waning and boosting of immunity proposed by \cite{barbarossa2017stability}, formulated as a system of equations with discrete and distributed delays. 
Bifurcation analyses for such models are scarce in the literature, due to the technicalities and complexities in the calculations. Our analysis was performed with a recently developed software release, MatCont~7p6 for MATLAB \citep{liessi2025matcont}.

For parameter values describing a highly infectious disease like pertussis, our analysis uncovered a much more complex dynamical picture than previous ones, including catastrophic bifurcations and parameter regions with bistability of different attractors. 
While bistability in waning-boosting models has been observed previously (for instance \cite{dafilis_frascoli_wood_mccaw_2012} and \cite{opoku-sarkodie_bifurcation_2024}), our work is the first to analyse a delayed model and find evidence of bistability not only of the endemic equilibrium and a limit cycle, but also of two distinct stable limit cycles differing in amplitude and period, for the same parameter values. 

Catastrophic bifurcations and bistability have practical relevance in public health, as disturbances of the system's state or parameters by non-pharmaceutical interventions can push the system to a different basin of attraction, hence potentially changing its long-term dynamics. Pertussis transmission, for instance, is highly influenced by waning-boosting dynamics, and post-pandemic outbreaks have sparked significant concern \citep{kang2024pertussis}.

Our work helps to better understand the role of duration of immunity and boosting force in shaping disease dynamics, further corroborating the evidence that waning-boosting infection dynamics is highly nonlinear, extremely complex, and public health interventions may have unintended consequences. 

\section*{Acknowledgments}
The research of FS was supported by the Engineering and Physical Sciences Research Council via the Mathematical Sciences Small Grant UKRI170: ``The dynamics of waning and boosting of immunity: new modelling and numerical tools''. FS is a member of the Computational Dynamics Laboratory (CDLab, University of Udine), of INdAM research group GNCS, of UMI research group Mo\-del\-li\-sti\-ca Socio-Epidemiologica, and of JUNIPER (Joint UNIversities Pandemic and Epidemiological Research).

The research of MP and GR was completed in the National Laboratory for Health Security, RRF-2.3.1-21-2022-00006. Additionally, their work was supported by the Ministry of Innovation and Technology of Hungary from the National Research, Development and Innovation Fund project no. TKP2021-NVA-09 and KKP 129877.

FS is grateful to Stefan Ruschel and Jan Sieber for helpful insights on the analysis with DDE-BIFTOOL. 

\section*{Data availability} 
No external data has been used to generate the results. The simulations are obtained using MATLAB 2023b and the free software packages MatCont version~7p6 and DDE-BIFTOOL version~3.2a. 

\bibliography{bibliography}

\end{document}